\documentclass[11pt]{amsart}

\usepackage{amsmath,amsthm,amssymb,fancyhdr,graphicx,bbm,cancel,color,mathrsfs,todonotes,hyperref,xypic, pinlabel}
\usepackage[all]{xy}

\usepackage{bm}

\newtheorem{thm}{Theorem}

\newtheorem{prop}[thm]{Proposition}

\theoremstyle{definition} 

\theoremstyle{remark} 

\newcommand{\dmo}{\DeclareMathOperator}

\newcommand{\R}{\mathbb{R}}
\newcommand{\Q}{\mathbb{Q}}
\newcommand{\co}{\mathbb{C}}\newcommand{\Z}{\mathbb{Z}}

\newcommand{\al}{\alpha}\newcommand{\ep}{\epsilon}\newcommand{\si}{\sigma}
\newcommand{\Om}{\Omega}\newcommand{\la}{\lambda}
\newcommand{\Ga}{\Gamma}
\newcommand{\wtil}{\widetilde}\newcommand{\Lam}{\Lambda}

\newcommand{\cd}{\cdots}\newcommand{\ld}{\ldots}
\newcommand{\sbs}{\subset}

\newcommand{\ra}{\rightarrow}

\newcommand{\bb}[1]{\mathbb{#1}}\newcommand{\un}[1]{\underline{#1}}\newcommand{\mf}{\mathfrak}

\newcommand{\fr}[2]{\frac{#1}{#2}}

\newcommand{\rt}{\rtimes}\newcommand{\ot}{\otimes}

\newcommand{\ti}{\times}

\dmo{\sgn}{sign}
\dmo{\we}{\wedge}
\dmo{\ind}{ind}\dmo{\Ind}{Ind}
\dmo{\bop}{\bigoplus}\dmo{\pic}{Pic}
\dmo{\coker}{coker}\dmo{\vol}{Vol}\dmo{\gal}{Gal}\dmo{\perm}{Perm}
\dmo{\tor}{Tor}\dmo{\ext}{Ext}\dmo{\Ext}{Ext}
\dmo{\aut}{aut}
\dmo{\Aut}{Aut}
\dmo{\inn}{Inn}\dmo{\var}{Var}
\dmo{\dep}{depth}
\dmo{\ad}{ad}\dmo{\curl}{curl}
\dmo{\hy}{\bb H}\dmo{\Sl}{SL}
\dmo{\SO}{SO}\dmo{\psl}{PSL}
\dmo{\isom}{Isom}\dmo{\Isom}{Isom}
\dmo{\conf}{Conf}
\dmo{\stab}{Stab}\dmo{\Jac}{Jac }
\dmo{\diam}{diam}\dmo{\fix}{Fixed}\dmo{\Fix}{Fix}
\dmo{\injR}{injRad}\dmo{\Ad}{Ad}
\dmo{\esv}{ess-vol}\dmo{\out}{Out}\dmo{\Out}{Out}
\dmo{\nil}{Nil}\dmo{\sol}{Sol}
\dmo{\Div}{div}
\dmo{\SU}{SU}
\dmo{\SP}{SP}
\dmo{\Sp}{Sp}
\dmo{\SL}{SL}
\dmo{\rk}{rk}
\dmo{\rank}{rank}
\dmo{\psp}{PSp}\dmo{\psu}{PSU}
\dmo{\PU}{PU}\dmo{\pgl}{PGL}
\dmo{\Mod}{Mod}\dmo{\range}{Range}
\dmo{\eu}{eu}\dmo{\mi}{mi}
\dmo{\Log}{Log}\dmo{\supp}{supp}
\dmo{\maps}{Maps}\dmo{\Gr}{Gr}
\dmo{\Pin}{Pin}
\dmo{\Spin}{Spin}\dmo{\Str}{Str}
\dmo{\Sq}{Sq}\dmo{\Symp}{Symp}
\dmo{\pd}{PD}\dmo{\PD}{PD}\dmo{\sig}{Sig}
\dmo{\Set}{Set}\dmo{\Top}{Top}
\dmo{\ev}{ev}\dmo{\St}{St}
\dmo{\Pt}{Pt}\dmo{\pt}{pt}

\dmo{\colim}{colim }\dmo{\Pl}{PL}
\dmo{\String}{String}\dmo{\smear}{smear}

\dmo{\dev}{dev}
\dmo{\met}{Met}\dmo{\contact}{Contact}
\dmo{\teich}{Teich}\dmo{\Teich}{Teich}\dmo{\qi}{QI}
\dmo{\der}{Der}
\dmo{\cl}{Cliff}\dmo{\Cl}{Cl}
\dmo{\Pf}{Pf}
\dmo{\ch}{ch}\dmo{\diag}{diag}
\dmo{\grad}{grad}\dmo{\Char}{char}
\dmo{\spec}{Spec}\dmo{\Arg}{Arg}
\dmo{\rad}{rad}\dmo{\im}{Im}
\dmo{\Hom}{Hom}\dmo{\End}{End}
\dmo{\tr}{tr}\dmo{\id}{Id}
\dmo{\gl}{GL}
\dmo{\sym}{Sym}\dmo{\Sym}{Sym}
\dmo{\com}{Comm}
\dmo{\Lk}{Lk}
\dmo{\CAT}{CAT}
\dmo{\Rep}{Rep}
\dmo{\Conf}{Conf}
\dmo{\PConf}{PConf}
\dmo{\Push}{Push}
\dmo{\Cont}{Cont}
\dmo{\sm}{\setminus}
\dmo{\vn}{\varnothing}
\dmo{\disk}{\mathbb D}
\dmo{\Trd}{Trd}\dmo{\Mat}{Mat}
\dmo{\Riem}{Riem}
\dmo{\Diffn}{\Diff_0}\dmo{\diff}{diff}
\dmo{\Diff}{Diff}\dmo{\homeo}{Homeo}
\dmo{\Homeo}{Homeo}\dmo{\Fr}{Fr}
\dmo{\rot}{rot}\dmo{\Emb}{Emb}
\dmo{\Ham}{Ham}\dmo{\Met}{Met}
\dmo{\Ein}{Ein}\dmo{\CP}{\co P}
\dmo{\Per}{Per}\dmo{\Ric}{Ric}
\newcommand{\C}{\mathbb C}\dmo{\Nrd}{Nrd}
\dmo{\Comp}{Comp}\dmo{\PSC}{PSC}
\dmo{\Cent}{Cent}\dmo{\Orb}{Orb}
\dmo{\aind}{a-ind}\dmo{\tind}{t-ind}
\dmo{\constant}{constant}
\dmo{\Td}{Td}
\dmo{\LMod}{LMod}
\dmo{\SMod}{SMod}
\dmo{\SDiff}{SDiff}
\dmo{\Br}{Br}
\dmo{\csch}{csch}
\dmo{\triv}{triv}
\dmo{\genus}{genus}
\dmo{\Homeq}{HomEq}
\dmo{\PP}{\mathbb{P}}
\dmo{\U}{U}
\dmo{\Gal}{Gal}
\dmo{\BDiff}{\wtil{\Diff}}
\dmo{\BAut}{\wtil{\Aut}}
\dmo{\Iso}{Iso}
\dmo{\codim}{codim}
\dmo{\II}{II}
\dmo{\I}{I}

\usepackage{enumerate,pdfpages,stmaryrd} 
\usepackage[margin=1in]{geometry}

\usepackage{tikz}
\usepackage{dsfont}
\usetikzlibrary{matrix}

\begin{document}

\title{Borel's stable range for the cohomology of arithmetic groups}

\author{Bena Tshishiku}
\address{Department of Mathematics, Harvard University, Cambridge, MA 02138} \email{tshishikub@gmail.com}


\date{\today}


\begin{abstract} 
In this note, we remark on the range in Borel's theorem on the stable cohomology of the arithmetic groups $\Sp_{2n}(\Z)$ and $\SO_{n,n}(\Z)$. This improves the range stated in Borel's original papers, an improvement that was known to Borel. Our main task is a technical computation involving the Weyl group action on roots and weights. 
\end{abstract}

\maketitle


\section{Introduction}

Let $G$ be a semi-simple algebraic group defined over $\Q$, and let $\Ga$ be a finite-index subgroup of $G(\Z)$. For $V$ an algebraic representation of $G$, Borel \cite{borel_cohoarith,borel_cohoarith2} computed the cohomology $H^i(\Ga;V)$ in a \emph{stable range}, i.e.\ for $i\le N$ for some constant $N=N(G,V)$ that depends only on $G,V$.

In some cases, the constant $N(G,V)$ that appears in \cite[\S9]{borel_cohoarith} and \cite{borel_cohoarith2} can be improved. This is remarked by Borel in \cite[\S3.8]{borel_cohoarith2}. In this note, we supply the details of Borel's remark when $G$ is one of the algebraic groups 
\[\SO_{n,n}=\{g\in \Sl_{2n}(\C): g^t\left(\begin{array}{cc}0&I\\I&0\end{array}\right)g=\left(\begin{array}{cc}0&I\\I&0\end{array}\right)\}\]
or 
\[\Sp_{2n}=\{g\in\Sl_{2n}(\C): g^t\left(\begin{array}{cc}0&I\\-I&0\end{array}\right)g=\left(\begin{array}{cc}0&I\\-I&0\end{array}\right)\}.\]

\begin{thm}[Borel stability for $\SO_{n,n}(\Z)$]\label{thm:so-stability}Fix $n\ge4$. 
Let $V$ be an irreducible rational representation of $\SO_{n,n}$, and let $\Ga<\SO_{n,n}(\Z)$ be a finite-index subgroup. If $k\le n-2$, then $H^k(\Ga;V)$ vanishes when $V$ is nontrivial, and agrees with the stable cohomology of $\SO_{n,n}(\Z)$ when $V$ is the trivial representation. 
\end{thm}

\begin{thm}[Borel stability for $\Sp_{2n}(\Z)$]\label{thm:sp-stability}
Fix $n\ge3$. Let $V$ be an irreducible rational representation of $\Sp_{2n}$, and let $\Ga<\Sp_{2n}(\Z)$ be a finite-index subgroup. If $k\le n-1$, then $H^k(\Ga;V)$ vanishes when $V$ is nontrivial, and agrees with the stable cohomology of $\Sp_{2n}(\Z)$ when $V$ is the trivial representation. 
\end{thm}

Theorem \ref{thm:sp-stability} is stated in \cite[Thm 3.2]{hain-infinitesimal} without proof. 

The cases $\SO_{2,2}$ and $\SO_{3,3}$ are exceptional because $\SO_{n,n}$ is isogenous to $\Sl_2\ti\Sl_2$ when $n=2$ and $\Sl_4$ when $n=3$. For $\SL_2(\Z)\ti\SL_2(\Z)$, the stable cohomology is trivial and there is no vanishing theorem. We remark further on the case of $\SO_{3,3}(\Z)$ in \S\ref{sec:so}. 

The bound in Theorem \ref{thm:so-stability} is nearly sharp. For example, when $n$ is odd, \cite{tshishiku-geocycles} proves that there is $\Ga<\SO_{n,n}(\Z)$ with $H^n(\Ga;\Q)\neq 0$, whereas if $i\le n-2$ is odd, then $H^{i}(\Ga;\Q)=0$ by Theorem \ref{thm:so-stability} and the determination of the stable cohomology of $\SO_{n,n}(\Z)$ \cite[\S11]{borel_cohoarith}.

This note originally appeared in the appendix of \cite{tshishiku-geocycles}. Theorems \ref{thm:so-stability} and \ref{thm:sp-stability} have been used by \cite{kupers-rw} in their study of the Torelli subgroup of diffeomorphisms of manifolds $\#_n(S^d\times S^d)$ when $d\ge3$. In this direction, we also mention that Theorems \ref{thm:so-stability} and \ref{thm:sp-stability} make the hypothesis on the degree of the representation $V$ in \cite[Prop.\ 3.9]{ebert-rw} unnecessary. 

\vspace{.1in} 
{\bf About the proof.} Theorems \ref{thm:so-stability} and \ref{thm:sp-stability} are deduced from the contents of \cite{borel_cohoarith} together with a representation-theoretic computation. 

We start by briefly summarizing Borel's approach to computing $H^*(\Ga; V)$ in a range; see also \cite{borel_cohoarith,borel_cohoarith2}. Fix a semi-simple algebraic group $G$ such that $G(\R)$ is of noncompact type, and let $X=G(\R)/K$ be the associated symmetric space. For a lattice $\Ga<G(\R)$, computing $H^*(\Ga;V)$ is equivalent to computing the homology of the complex $\Om^*(X;V)^\Ga$ of $V$-valued, $\Ga$-invariant differential forms on $X$. The subcomplex $I^*_{G,V}\subset\Om^*(X;V)^\Ga$ of $G(\R)$-invariant forms consists of closed forms, so there is a homomorphism 
\[j:I^*_{G,V}\ra H^*(\Ga;\R),\]
whose image is known as the \emph{stable cohomology}. The ring $I^*_{G,V}$ is easily computed: it is isomorphic to $H^*(X_u;V)$, where $X_u$ is the compact symmetric space dual to $X$, and it is also identified with Lie algebra cohomology $H^*(\mf g,K;V)$. In particular, if $V$ is irreducible and nontrivial, then $H^{*}(\mf g,K;V)$ is trivial \cite[Ch.\ II, Cor.\ 3.2]{borel-wallach}. Borel showed that $j^*$ is bijective in a range $i\le\min \{m(G(\R)),c(G,V)\}$. See \cite[Thm.\ 7.5]{borel_cohoarith} and \cite[Thm.\ 4.4]{borel_cohoarith2}.

To apply Borel's theorem, one wants to understand the constants $m(G(\R))$ and $c(G,V)$. According to \cite[\S4]{borel_cohoarith2}, $m(G(\R))\ge\rk_\R G(\R)-1$ for every $G$ that is almost simple over $\R$ (this includes $\SO_{n,n}$ and $\Sp_{2n}$, both of which have rank $n$). The constant $c(G,V)$ can be computed with some representation theory. 

\un{The constant $c(G,V)$}. Let $\mf g$ be the Lie algebra of $G(\C)$, and let $B\sbs G(\C)$ be a minimal parabolic (i.e.\ Borel) subgroup with Levi decomposition $B=U\rt A$. Let $\mf a$ and $\mf u$ be the corresponding Lie algebras. Here $\mf a\sbs\mf g$ is a maximal abelian (i.e.\ Cartan) subalgebra. The weights of $\mf a$ acting on $\mf g$ are called the \emph{roots} of $\mf g$, and the subset of weights of $\mf a$ acting on $\mf u$ are called \emph{positive}. Let $\rho$ be half the sum of the positive roots. A positive root is called \emph{simple} if it cannot be expressed as a nontrivial sum of positive roots. The simple positive roots $\{\al_k\}$ form a basis for $\mf a^*$. An element $\phi\in\mf a^*$ is called \emph{dominant} (resp.\ \emph{dominant regular}), denoted $\phi\ge0$ (resp. $\phi>0$), if $\phi=\sum c_k\>\al_k$ with $c_k\ge0$ (resp.\ $c_k>0$) for each $k$. 

Borel's constant $c(G,V)$ is the largest $q$ so that $\rho+\mu>0$ for every weight $\mu$ of $\Lam^q\mf u^*\ot V$, c.f.\ \cite[\S2 and Thm.\ 4.4]{borel_cohoarith}.

\un{A better constant $C(G,V)$.} According to \cite[Rmk.\ 3.8]{borel_cohoarith2} (see also \cite[Thm.\ 3.1]{garland-hsiang} and \cite[(3.20) and (4.57)]{zucker}), there is a better constant $C(G,V)\ge c(G,V)$ so $j^*$ bijective in degrees $i\le\{m(G(\R)),C(G,V)\}$. To define this constant, let $W$ be the Weyl group of $G(\C)$. For each $q\ge0$, let $W^q\sbs W$ be the subset of elements that send exactly $q$ positive roots to negative roots. Denoting the highest weight of $V$ by $\la$, define
\[C(G,V)=\max\{q: \si(\rho+\la)>0\text{ for all }\si\in W^q\}.\]

As Borel remarks, $C(G,V)$ can be interpreted as the largest $q$ for which $\rho+\mu>0$ for every weight $\mu$ of $H^q(\mf u;V)$. Since the Lie algebra cohomology $H^*(\mf u;V)$ is the homology of the complex $\Lam^*\mf u^*\ot V$, it follows that the weights of the former are a subset of the weights of the latter, so $c(G,V)\le C(G,V)$. 

In the remainder of this note, we compute the value of $C(G,V)$ when $G$ is $\Q$-split of type $C_n$ or $D_n$, i.e.\ $G(\Z)$ is commensurable with $\Sp_{2n}(\Z)$ or $\SO_{n,n}(\Z)$. 

\vspace{.1in} 
{\bf Acknowledgements.} The author would like to thank R.\ Hain for helpful email correspondence. 

\section{Computation for $\SO_{n,n}$} \label{sec:so}

The main goal of this section is to prove the following proposition. 

\begin{prop}\label{prop:so-constant}Fix $n\ge4$, and let $G=\SO_{n,n}$. Then $C(G,V)\ge n-2$ for each irreducible finite dimensional rational representation $V$ of $G$. 
\end{prop}

This is divided into two steps: we first show $C(G,\C)=n-2$ for the $\C$ the trivial representation (Proposition \ref{prop:so-constant-triv}), and then we show $C(G,V)\ge C(G,\C)$ for any other representation (Proposition \ref{prop:so-other}). 

To begin, we need the following information from \cite[pg.\ 256-258]{bourbaki}. Below $\ep_1,\ldots,\ep_n$ are the standard coordinate functionals on $\mf a\sbs\mf g$. 
\begin{itemize}
\item The simple roots are $\al_1=\ep_1-\ep_2$, $\ld$, $\al_{n-1}=\ep_{n-1}-\ep_n$, and $\al_n=\ep_{n-1}+\ep_n$.
\item The half-sum of positive roots is $\rho=\sum_{i=1}^nr_i\>\al_i$, where $r_i=\fr{(2n-i-1)i}{2}$ for $1\le i\le n-2$ and $r_{n-1}=r_n=\fr{n(n-1)}{4}$. 
\item The Weyl group $W=(\Z/2\Z)^{n-1}\rt S_n$ acts as the even signed permutation group of $\{\pm \ep_1,\ld,\pm\ep_n\}$, i.e.\ the symmetric group $S_n$ acts by permuting the indices of $\ep_1,\ld,\ep_n$, and $(\Z/2\Z)^{n-1}$ acts by an even number of sign changes.
\end{itemize} 

Let $\tau_i\in W$ be the reflection fixing the orthogonal complement of $\al_i$ (with respect to the inner product where the $\ep_i$ are othonormal). The $\tau_i$ generate $W$, and we write
\[S=\{\tau_1,\ldots,\tau_n\}.\] The action of $\tau_i$ on the roots sends $\al_i$ to $-\al_i$ and permutes the remaining positive roots. Thus $\tau_i\in W^1$, and it's not hard to show that $\si\in W^q$ if and only if the word length of $\si$ with respect to $S$ is $q$. See \cite[\S10.3]{humphreys-intro-lie} for details. 

In what follows we will work in the basis $(\ep_1,\ldots,\ep_n)$ instead of $(\al_1,\ldots,\al_n)$. We record how these two bases are related: if $\sum x_i\al_i=\sum y_i\ep_i$, then 
\begin{equation}\label{eqn:basis-change}
\begin{array}{rclcl}
x_k&=&y_1+\cd+y_k&k\le n-2\\[2mm]
x_{n-1}&=&\fr{1}{2}(y_1+\cd+y_{n-1}-y_n)\\[2mm]
x_n&=&\fr{1}{2}(y_1+\cd+y_{n-1}+y_n)
\end{array}
\end{equation}
For $i=1,\ld,n-1$, the reflection $\tau_i$ interchanges $\ep_i$ and $\ep_{i+1}$ (and acts trivially on the remaining $\ep_j$), while $\tau_n$ interchanges $\ep_{n-1}$ and $\ep_n$ and changes their signs. In $\ep_i$-coordinates,
\[\rho=(n-1,n-2,\ldots,2,1,0).\]

\begin{prop}\label{prop:so-constant-triv}
Fix $n\ge3$, and let $G=\SO_{n,n}$. Then $C(G,\C)=n-2$. 
\end{prop}
\begin{proof} First observe that the image of $\rho$ under $\si=\tau_1\cd\tau_{n-1}$ is not dominant regular. Indeed in $\ep_i$-coordinates, $\si(\rho)=(0,n-1,n-2,\ldots,2,1)$, which is not dominant regular since the coefficient on $\al_1$ is 0.  This implies $C(G,\C)\le n-2$. 

It remains to show $C(G,\C)\ge n-2$, i.e.\ if $\si\in W$ can be expressed as a word in $S$ of length $\ell\le n-2$, then $\si(\rho)$ is dominant regular. Recall above that the $\tau_i$ act on $\ep_i$-coordinates as signed permutations, so the coordinates of $\si(\rho)=(y_1,\ldots,y_n)$ are a signed permutation of the coordinates of $\rho=(n-1,\ldots,1,0)$. In order to show $\si(\rho)>0$, we need to show each of the sums $y_1+\cd+y_k$ is positive for $k\neq n-1$ and also that $y_1+\cd+y_{n-1}-y_n$ is positive. 

We first consider two special cases from which the general case follows. 

{\it Special case $1$.} Suppose that $\si$ is a word in $S\setminus\{\tau_{n-1},\tau_n\}$. In $\ep_i$-coordinates $\tau_1,\ldots,\tau_{n-2}$ act as permutations without sign changes that fix the last coordinate, so the coordinates of $\si(\rho)=(y_1,\ldots,y_{n-1},0)$, where $(y_1,\ld,y_{n-1})$ is a permutation of $(n-1,\ldots,1)$. In particular, $y_1,\ld,y_{n-1}$ are all positive, and it follows that $\si(\rho)$ is regular dominant. 

{\it Special case $2$.} Suppose that $\si$ is a word in $S\setminus\{\tau_1\}$. Then $\si(\rho)=(n-1,y_2,\ldots,y_n)$, where $(y_2,\ldots,y_n)$ is a signed permutation of $(n-2,\ldots,1,0)$.

Since $\tau_n$ is the only element of $S$ that changes any sign, in order for $j$ (the $(n-j)$-th coordinate of $\rho$) to appear with a negative sign in $\si(\rho)$, the length of $\si$ must be at least $j$ (this follows immediately from the $\tau_i$ action on the coordinates; note, for example, that the sign of $j$ in $\tau_n\tau_{n-2}\cd\tau_{n-j}(\rho)$ is negative). Similarly, in order for each of the coefficients $j_1,\ldots,j_m$ of $\rho$ to appear with negative signs in $\si(\rho)$, the length of $\si$ must be at least $j_1+\cdots+j_m$ (again this follows from the $\tau_i$ action; note that each $\tau_i$ only moves one coefficient to the right at a time). Let $j_1,\ldots,j_m$ be the coefficients of $\rho$ that become negative in $\si(\rho)$. Then $j_1+\cdots+j_m\le n-2$ because $\si$ has length $\le n-2$. Hence for $1\le i\le n$,
\[y_1+\cdots+y_i\ge (n-1)-(j_1+\cdots+j_m)\ge (n-1)-(n-2)>0.\]
It follows that the coefficient of $\al_i$ in $\si(\rho)$ is positive for each $i$, possibly with the exception of $i=n-1$ (c.f.\ (\ref{eqn:basis-change})). By the same reasoning, the coefficient of $\al_{n-1}$ is also positive: let $j_1,\ldots,j_m$ be the coefficients of $\rho$ that are negative in $\si(\rho)$, and suppose $y_n=j_{m+1}$. Moving $j_{m+1}$ to the $n$-th coordinate requires a word of length $j_{m+1}$ (e.g.\ $\tau_{n-1}\tau_{n-2}\cd\tau_{n-j}$), so as above $j_1+\cdots+j_{m+1}\le n-2$, and so, similar to the above, 
\[y_1+\cdots+y_{n-1}-y_n\ge (n-1)-(j_1+\cdots+j_{m+1})>0.\]

{\it General case.} Suppose $\si$ is any word in $\tau_1,\ld,\tau_n$ of length $\le n-2$. Then $\tau_i$ does not appear in $\si$ for some $1\le i\le n-1$, and we can write $\si=\si_1\si_2$, where $\si_1$ is a word in $\{\tau_1,\ldots,\tau_{i-1}\}$ and $\si_2$ is a word in $\{\tau_{i+1},\ld,\tau_n\}$. For $j\le i$, the coefficient of $\al_j$ in $\si(\rho)$ and $\si_1(\rho)$ agree, and for $j\ge i+1$, the coefficient of $\al_j$ in $\si(\rho)$ and $\si_2(\rho)$ agree, so $\si(\rho)$ is dominant regular by the previous two cases. 
\end{proof} 

\begin{prop}\label{prop:so-other}
Fix $n\ge4$, and let $G=\SO_{n,n}$. If $V$ is an irreducible representation, then $C(G,V)\ge C(G,\C)$. 
\end{prop}

\begin{proof}
Let $\la$ be the highest weight of $V$. According to \cite[\S19.2]{fh}, $\la$ can be expressed as an integral linear combination $\la=\sum_{k=1}^n a_k\phi_k$, where $a_k\ge0$ and 
\begin{equation}\label{eqn:highest-weight-so}\phi_k=\left\{\begin{array}{lllll}\ep_1+\cd+\ep_k&k\le n-2\\[2mm]
(\ep_1+\cd+\ep_{n-1}-\ep_n)/2&k=n-1\\[2mm]
(\ep_1+\cd+\ep_n)/2&k=n.
\end{array}\right.
\end{equation}
If $\si\in W$, then $\si(\rho+\la)=\si(\rho)+\sum_k a_k\>\si(\phi_k)$. We proceed by studying when $\si(\phi_k)$ is dominant. To show $C(G,V)\ge C(G,\C)=n-2$, it suffices to show that if $\si\in W^q$ for $q\le n-2$, then $\si(\phi_k)\ge0$ for each $1\le k\le n$. Then for any highest weight $\la=\sum a_k\phi_k$, we conclude that $\si(\rho+\la)=\si(\rho)+\sum a_k\>\si(\phi_k)$ is dominant regular because $\si(\phi_k)\ge0$ and $\si(\rho)>0$ (Proposition \ref{prop:so-constant-triv}). 

We consider separately cases $1\le k\le n-2$ and $k=n-1,n$. In either case the argument is similar to the corresponding step in the proof of Proposition \ref{prop:so-constant-triv}. 

Fix $1\le k\le n-2$ and write $\phi=\phi_k$. In $\ep_i$-coordinates, $\phi=(1,\ldots,1,0,\ldots,0)$. Next we bound from below the minimum word length of $\si$ needed for $\si(\phi)<0$, and we will find that there is no $\si$ of length $\le n-2$. 

First observe that the only way to act by elements of $S$ to make a coefficient of $\phi$ negative is to move that coefficient to the right (using a word like $\tau_{n-2}\cd\tau_i$), and then apply $\tau_n$. Therefore, fixing $\ell<k/2$, any word $\si$ such that $\si(\phi)$ has $\ell+1$ negative coordinates has length at least
\begin{equation}\label{eqn:so-word-bound}
(n-k)+\cdots+(n-k+\ell)=n(\ell+1)-\left[\fr{(k+1)k}{2}-\fr{(k-\ell)(k-\ell-1)}{2}
\right].\end{equation}

\begin{figure}
{\color{blue}$\phi=(1,\ldots,1,\underbrace{1,\ldots,1}_{\ell+1},\underbrace{0,\ldots,0}_{n-k})$}
\caption{To make $\ell+1$ positive coefficients of $\phi$ negative requires a word whose length is at least the quantity in (\ref{eqn:so-word-bound}).}
\label{fig:1}
\end{figure}
See Figure \ref{fig:1}. 

After creating $\ell+1$ negative coefficients, to make a non-dominant vector, one needs to move sufficiently many positive entries to the right, passed the negative entries. Since we start with $k=\ell+(k-2\ell-1)+(\ell+1)$ positive entries, we must move $(k-2\ell-1)$ positive entries passed the $(\ell+1)$ negative entries. This requires a word of length at least 
\begin{equation}\label{eqn:so-word-bound2}(k-2\ell-1)(\ell+1).\end{equation}

\begin{figure}
{\color{blue}$(\underbrace{1,\ldots,1}_\ell,\underbrace{1,\ldots,1}_{k-2\ell-1},\underbrace{0,\ldots,0}_{n-k},\underbrace{-1,\ldots,-1}_{\ell+1})$}
\caption{We can make this vector non-dominant by moving $k-2\ell-1$ positive entries past $\ell+1$ negative entries. This requires a word whose length is at least the quantity in (\ref{eqn:so-word-bound2}).}
\label{fig:2}
\end{figure}
See Figure \ref{fig:2}. 

Now we conclude. Suppose for a contradiction that $\si$ has length $\le n-2$ and that $\si(\phi)<0$. Write $\si(\phi)=(y_1,\ldots,y_n)$, and let $i$ be the smallest index so that the coefficient of $\al_i$ in $\si(\phi)$ is negative. If $i\neq n-1$, then this means $y_1+\cdots+y_i<0$. We will assume $i\neq n-1$; the case $i=n-1$ is similar (c.f.\ the proof of Proposition \ref{prop:so-constant-triv}). The terms in this sum $y_1+\cd+y_i$ are all $+1, 0, -1$. By replacing $\si$ with a shorter word, we can assume that the summands occur in decreasing order $1+\cd+1+0+\cd+0+-1+\cd+-1$ (this follows from the description of the $\tau_i$ action and the fact that the coefficients of $\phi$ are decreasing). By minimality of our choice of $i$, if there are $\ell$ positive terms in the sum, then there are $\ell+1$ negative terms. Note then that $2\ell+1\le k$, so $\ell<k/2$. 

Combining (\ref{eqn:so-word-bound}) and (\ref{eqn:so-word-bound2}), if the leading coefficients of $\si(\phi)$ are $(1,\ld,1,0,\ld,0,-1,\ld,-1,\ld)$, then the length of $\si$ is at least 
\[n(\ell+1)-\left[\fr{(k+1)k}{2}-\fr{(k-\ell)(k-\ell-1)}{2}\right]+(k-2\ell-1)(\ell+1)=n(\ell+1)-\fr{3\ell^2+5\ell+2}{2}.\]
Since we're assuming $\si$ has length $\le n-2$, we must have $n(\ell+1)-\fr{3\ell^2+5\ell+2}{2}\le n-2$. This inequality implies that 
\[n\le \fr{3\ell+5}{2}.\]
Since $\ell<k/2\le (n-2)/2$, this implies that $n<4$. This is contrary to our hypothesis, so we conclude that there does not exist $\si$ of length $\le n-2$ so that $\si(\phi)<0$. 

The same analysis can be applied to $\phi_{n-1}$ and $\phi_n$. The details are unilluminating and can easily be supplied by the interested reader, so we omit them here. 
\end{proof}

Proposition \ref{prop:so-other} is false for $n=3$. In this case, $C(G,\C)=1$, but there are $V$ with $C(G,V)=0$. For example, take $V$ the irreducible representation with highest weight $m\phi_2$. Observe that, in $\ep_i$-coordinates, 
\[\tau_2(\rho+m\phi_2)=\left(2+\fr{m}{2}, -\fr{m}{2},1+\fr{m}{2}\right),\]
so the coefficient of $\al_2$ in $\tau_2(\rho+m\phi_2)$ is $\fr{1}{2}\left((2+\fr{m}{2})-\fr{m}{2}-(1+\fr{m}{2})\right)=1-\fr{m}{2}$, which is non-positive if $m\ge2$. This implies that $C(G,V)=0$, and Borel's theorem does not allow one to conclude, for example, that $H^1(\Ga;V)=0$ for a lattice $\Ga<\SO_{3,3}(\Z)$. However, $H^1(\Ga;V)$ does vanish for any nontrivial $V$ by a theorem of Margulis \cite[Ch.\ VII, Cor.\ 6.17]{margulis}. 

The failure of Proposition \ref{prop:so-other} in the case $n=3$ is related to the fact that $\SO_{3,3}$ is isogenous to $\SL_4$. For $\SL_{n+1}$ one can compute that $C(G,\C)$ is the smallest integer strictly less than $n/2$, but it's not true the $C(G,V)\ge C(G,\C)$ for every irreducible representation. Indeed if one takes $V=\Sym^m(\C^{n+1})$, then $C(G,V)=0$ for $m$ sufficiently large. In this direction we remark that there are other known vanishing results for $H^*(\Ga;V)$. See \cite[pg.\ 143]{li-schwermer}. 

Proposition \ref{prop:so-other} gives a lower bound on $C(G,V)$. We remark on the upper bound. Observe that if $\si=\tau_1\cd\tau_n$, then $\si(\phi)\le0$ because the coefficient of $\al_1$ is non-positive. Since the coefficient of $\al_1$ in $\si(\rho)$ is also non-positive, it follows that $\si(\rho+\la)\le0$ for every highest weight $\la$. This shows that $C(G,V)\le n-1$, and so 
\[n-2\le C(G,V)\le n-1\] 
for any irreducible $V$. For any particular $V$ one can determine which inequality is strict. For example $C(G,V)=n-2$ when the highest weight of $V$ is one of the basis vectors $\phi_1,\ldots,\phi_n$  (to show $C(G,V)\le n-2$, consider $\si=\tau_1\cd\tau_{n-1}$ if $1\le k\le n-1$ and $\si=\tau_1\cd\tau_{n-2}\tau_n$ for $k=n$). We leave further computations to the reader.  

\section{Computation for $\Sp_{2n}(\R)$} \label{sec:sp}

In this section we carry out the analysis of \S\ref{sec:so} for $\Sp_{2n}$. The goal is to prove the following proposition. 

\begin{prop}\label{prop:sp-constant}Fix $n\ge3$, and let $G=\Sp_{2n}$. Then $C(G,V)= n-1$ for each irreducible finite dimensional rational representation $V$ of $G$. 
\end{prop}

The outline of the argument is similar to the argument for Proposition \ref{prop:so-constant}. We explain the main differences and refer the reader to \S\ref{sec:so} when the details are similar. We start with the following information is from \cite[pg.\ 254-255]{bourbaki}. 
\begin{itemize}
\item The simple roots are $\al_1=\ep_1-\ep_2$, $\ld$, $\al_{n-1}=\ep_{n-1}-\ep_n$, and $\al_n=2\ep_n$. 
\item The half the sum of positive roots is $\rho=\sum r_i\>\al_i$, where $r_i=\fr{(2n-i+1)i}{2}$ for $1\le i\le n-1$ and $r_n=\fr{n(n+1)}{4}$. 
\item The Weyl group $W=(\Z/2\Z)^n\rt S_n$. It acts as the signed permutation group of $\{\pm\ep_1,\ld,\pm\ep_n\}$. 
\end{itemize}

Let $\tau_i\in W$ be the reflection fixing the orthogonal complement of $\al_i$. For $1\le i\le n-1$, the reflection $\tau_i$ interchanges $\ep_i$ and $\ep_{i+1}$, while $\tau_n$ only changes the sign on $\ep_n$. The set $S=\{\tau_1,\ldots,\tau_n\}$ generates $W$. As in the $\SO_{n,n}$ case, $S\sbs W^1$ and $\si\in W^q$ if and only if the $S$ word-length of $\si$ is $q$. 

We record how the bases $(\ep_1,\ldots,\ep_n)$ and $(\al_1,\ld,\al_n)$ are related: if $\sum x_i\al_i=\sum y_i\ep_i$, then 
\begin{equation}\label{eqn:sp-basis-change}
\begin{array}{rclcl}
x_k&=&y_1+\cd+y_k&k\le n-1\\[2mm]
x_n&=&\fr{1}{2}(y_1+\cd+y_{n-1}+y_n)
\end{array}
\end{equation}
In $\ep_i$-coordinates, 
\[\rho=(n,n-1,\ldots,2,1).\]

\begin{prop}\label{prop:sp-constant-triv} If $G=\Sp_{2n}$, then $C(G,\C)=n-1$.
\end{prop}

\begin{proof} 
First observe that if $\si=\tau_1\cd\tau_n$, then $\si(\rho)=(-1,n,\ldots,2)$ is not dominant regular. This shows $C(G,\C)\le n-1$. 

To show $C(G,\C)\ge n-1$, let $\si$ be a word in $S$ of length $\le n-1$. We will show $\si(\rho)$ is dominant regular. 

{\it Special case $1$.} First consider the case that $\si$ is a word in $S\setminus\{\tau_n\}$. Since $\tau_1,\ldots,\tau_{n-1}$ act as permutations without changing sign, $\si(\rho)=(y_1,\ldots,y_{n-1},1)$, where $(y_1,\ldots,y_{n-1})$ are a permutation of $(n,\ldots,2)$. Then $y_1+\cd+y_i>0$ for each $1\le i\le n$, which implies that $\si(\rho)$ is dominant regular. 

{\it Special case $2$.} Next consider the case that $\si$ is a word in $S\setminus\{\tau_1\}$. Then $\si(\rho)=(n,y_2,\ldots,y_n)$, where $(y_2,\ldots,y_n)$ is a signed permutation of $(n-1,\ldots,1)$. 

Since $\tau_n$ is the only element of $S$ that changes any sign, in order for $j$ (the $(n-j+1)$-st coordinate of $\rho$) to appear with a negative sign in $\si(\rho)$, the length of $\si$ must be at least $j$. Similarly, in order for each of the coefficients $j_1,\ldots,j_m$ of $\rho$ to appear with negative signs in $\si(\rho)$, the length of $\si$ must be at least $j_1+\cdots+j_m$. Let $j_1,\ldots,j_m$ be the coefficients of $\rho$ that become negative in $\si(\rho)$. Then $j_1+\cdots+j_m\le n-1$ because $\si$ has length $\le n-1$. Hence for $1\le i\le n$,
\[y_1+\cd+y_i\ge n-(j_1+\cdots+j_m)\ge n-(n-1)>0.\]
This shows that $\si(\rho)$ is dominant regular. 

{\it General case.} If $\si$ has length $\le n-1$, then there is some index $1\le i\le n$ so that $\tau_i$ is not in $\si$. We covered the cases $i=1$ and $i=n$ above, so we can assume $1<i<n$. Then we can write $\si=\si_1\si_2$, where $\si_1$ is a word in $\{\tau_1,\ld,\tau_{i-1}\}$ and $\si_2$ is a word in $\{\tau_{i+1},\ld,\tau_n\}$. Then the coefficients of $\al_j$ in $\si_1(\rho)$ and $\si(\rho)$ agree for $j\le i$ and the coefficients of $\al_j$ in $\si_2(\rho)$ and $\si(\rho)$ agree for $j\ge i+1$, so we again reduce to the previous cases to conclude that $\si(\rho)$ is dominant regular. 
\end{proof}

\begin{prop}
Let $V$ be an irreducible representation. Then $C(G,V)= C(G,\C)$. 
\end{prop}

\begin{proof}
Let $\la$ be the highest weight of $V$. According to \cite[\S17.2]{fh}, $\la$ can be expressed as an integral linear combination $\la=\sum_{k=1}^n a_k\>\phi_k$, where $a_k\ge0$ and $\phi_k=\ep_1+\cd+\ep_k$. If $\si\in W$, then $\si(\rho+\la)=\si(\rho)+\sum_k a_k\>\si(\phi_k)$. The proof of the proposition will follow by studying $\si(\phi_k)$.

First we explain why $C(G,V)\le C(G,\C)=n-1$. Observe that for $\si=\tau_1\cd\tau_n$, the coefficient of $\al_1$ in $\si(\phi_k)$ is 0 for each $1\le k\le n$. In Proposition \ref{prop:sp-constant-triv} we showed that the coefficient of $\al_1$ in $\si(\rho)$ is negative, so it follows that $\si(\rho+\la)\le0$. Hence $C(G,V)\le n-1$.

To show that $C(G,V)\ge n-1$, it suffices to show that if $\si\in W^q$ for $q\le n-1$, then $\si(\phi_k)\ge0$ for each $1\le k\le n$. To simplify the notation, fix $k$ and write $\phi=\phi_k$. In $\ep_i$-coordinates $\phi=(1,\ldots,1,0,\ldots,0)$. Next we bound from below the minimum word length of $\si$ needed for $\si(\phi)<0$, and we will find that there is no $\si$ of length $\le n-1$. 

First observe that the only way to act by elements of $S$ to make a coefficient of $\phi$ negative is to move that coefficient to the right (using a word like $\tau_{n-1}\cd\tau_i$) and the apply $\tau_n$. Therefore, fixing $\ell<k/2$, any word $\si$ such that $\si(\phi)$ has $\ell+1$ negative coordinates has length at least 
\begin{equation}\label{eqn:sp-word-bound}
(n-k+1)+\cdots+(n-k+1+\ell)=n(\ell+1)-\left[\fr{k(k-1)}{2}-\fr{(k-\ell-1)(k-\ell-2)}{2}\right]. 
\end{equation} 
After creating $\ell+1$ negative coefficients, to make a non-dominant vector, one needs to move sufficiently many positive entries to the right, passed the negative entries. Since we start with $k=\ell+(k-2\ell+1)+(\ell+1)$ positive entries, we must move $(k-2\ell-1)$ positive entries passed the $(\ell+1)$ negative entries. This requires a word of length at least 
\begin{equation}\label{eqn:sp-word-bound2}
(k-2\ell-1)(\ell+1).
\end{equation} 

Now we conclude. Suppose for a contradiction that $\si$ has length $\le n-1$ and that $\si(\phi)<0$. Write $\si(\phi)=(y_1,\ldots,y_n)$, and let $i$ be the smallest index so that the coefficient of $\al_i$ in $\si(\phi)$ is negative. Then $y_1+\cdots+y_i<0$. The terms in the sum $y_1+\cd+y_i$ are all $+1,0,-1$. By replacing $\si$ with a shorter word, we can assume that the summands occur in decreasing order. By minimality of our choice of $i$, if there are $\ell$ positive terms in the sum, then there are $\ell+1$ negative terms. Then $2\ell+1\le k$, so $\ell<k/2$. 

Combining (\ref{eqn:sp-word-bound}) and (\ref{eqn:sp-word-bound2}), if the leading coefficients of $\si(\phi)$ are $(1,\ldots,1,0,\ldots,0,-1,\ldots,-1,\ldots)$, then the length of $\si$ is at least 
\[n(\ell+1)-\left[\fr{k(k-1)}{2}-\fr{(k-\ell-1)(k-\ell-2)}{2}\right]+(k-2\ell-1)(\ell+1)=n(\ell+1)-\fr{3\ell^2+3\ell}{2}.\]

Since we're assuming $\si$ has length $\le n-1$, we must have $n(\ell+1)-\fr{3\ell^2+3\ell}{2}\le n-1$. This inequality implies that 
\[n\le \fr{3\ell+3}{2}-\fr{1}{\ell}\le\fr{3\ell+3}{2}.\]
Since $\ell<k/2$, if $k\le n-1$, this implies that $n<3$, which contradicts the hypothesis. If $k=n$, then we can only conclude $n< 6$. 

Assume now that $k=n$ and $n\le 5$. Since $\ell<k/2=n/2$ this implies that either $\ell=1$ and $3\le n\le 5$ or $\ell=2$ and $n= 5$. The inequality $n\le \fr{3\ell+3}{2}-\fr{1}{\ell}$ implies that $n\le 2$ when $\ell=1$ and it implies $n\le 4$ when $\ell=2$. In either case, this is a contradiction. Therefore, we conclude that if $\si$ has length $\le n-1$, then $\si(\phi)\ge0$. This completes the proof. 
\end{proof}

\bibliographystyle{alpha}
\bibliography{borel.bib}

\end{document}